\documentclass{amsart}
\usepackage{amsmath}
\usepackage{amsfonts}
\usepackage{amssymb}
\usepackage[latin2]{inputenc}
\usepackage{graphicx}
\usepackage{array}
\DeclareMathOperator{\co}{\mathbb C}

\DeclareMathOperator{\ZZ}{\it (I_{n}+ZZ^{*})}

\newtheorem{theorem}{Theorem}[section]
\newtheorem{remark}{Remark}[section]
\newtheorem{proposition}{Proposition}[section]

\theoremstyle{definition}

\numberwithin{equation}{section} 

\begin{document}

\title[representative coordinates]{On the Bergman representative coordinates}
\author{Żywomir Dinew}
\address{Jagiellonian University, Institute of Mathematics,
Łojasiewicza 6, 30-348 Kraków, Poland  }
\email{Zywomir.Dinew@im.uj.edu.pl}
\thanks{The author was supported  by the Polish Ministry of Science and Higher Education Grant N N201 271235.}
\begin{abstract}
We study the set where the so-called Bergman representative coordinates (or Bergman functions) form an immersion. We provide an estimate of the size of a maximal geodesic ball with respect to the Bergman metric, contained in this set. By  concrete examples we show that these estimates are the best possible. 
\end{abstract}
\maketitle
\begin{section}{Introduction}
Bergman representative coordinates were introduced by Bergman in
\cite{MR1512585} as a tool in his program of generalizing the Riemann mapping
theorem to $\co^{n}, n>1$.

Their usefulness is based (among others) on the fact that 
biholomorphic mappings become linear when represented in these coordinates (See eg.\cite{MR1228447}).

It is hard to work with these coordinates mainly because they are
not defined globally even in the domain case, Nevertheless some remarkable results were obtained by using them.
Lu Qi-Keng \cite{MR0206990}  proved that
 any domain with complete Bergman metric of constant negative
 holomorphic sectional curvature is biholomorphic to the unit
 ball in $\co^{n}$.

The so-called Bergman representative coordinates, respective to a point $z_{0}$ are:

$$w_{i}(z)=\sum_{j=1}^{n}T^{\overline{j}i}(z_{0})\frac{\partial}{\partial \bar{\zeta_{j}}}\log {\frac{K(z,\zeta)}{K(\zeta,\zeta)}}_{|\zeta=z_{0}},$$
where $K(z,\zeta)$ is the Bergman kernel of the domain $\Omega$ and $T^{\overline{j}i}(z_{0})$ is the inverse  matrix of the matrix $(T_{i\overline{j}}(z))_{i,j=1..n}=\left(\frac{\partial ^{2}}{\partial z_{i}\partial \bar z_{j}}\log K(z,z)\right)_{i,j=1..n}$, evaluated at the point $z=z_{0}$ (we refer to section \ref{section1} for all the definitions).

Looking at the definition one immediately comes upon two issues:
Are the above expressions well defined?
 Are they indeed coordinates?

 Clearly in a small neighborhood of the point $z_{0}$ the answers to
 both questions are affirmative since $\frac{\partial (w_{1},..,w_{n})}{\partial(z_{1},..,z_{n})}_{|z=z_{0}}$ is the identity matrix.

 Concerning the first question, one immediately sees that the only
 possible obstruction which may appear is that
 $K(z,\zeta)$ may have zeros. This is the reason for which studying zeros
 of the Bergman kernel attracted so much interest. Domains for
 which the Bergman kernel is zero-free are known as domains
 satisfying the Lu Qi-Keng conjecture or just Lu Qi-Keng domains. From nowadays perspective
 it is known that virtually all (in a sense, see \cite{MR1317032}) domains are not Lu
 Qi-Keng domains. On the other hand it is clear that for fixed $z_{0}$
 (as is in our case) the zero set of the Bergman kernel will be an
 analytic set and hence $w_{i}$ are well defined almost everywhere (on an open dense subset) in
 $\Omega$. This topological information is one of the main ingredients in the Lu Qi-Keng's argument \cite{MR0206990}. On the other hand for many geometric problems just
 topological information on the domain of definition (since one already knows that there
 is no hope to define the coordinates globally in general, but a local definition is at hand)
  is not enough. One would like to know whether there are subdomains $\Omega'_{z_{0}}$ of these domains of definition, which are related not to the topology but to the geometry of $\Omega$. In particular one would like to know "how small" these
  neighborhoods of $z_{0}$ in which $w_{i}$ are well defined must be, can one control them in a reasonable way
  (e.g., with dependence on the geometry of $\Omega$)
 when the point $z_{0}$ is perturbed?

 Even if well defined, the representative coordinates would have
 been useless if they do not yield a basis of local vector-fields,
 i.e.,
\begin{equation}\label{jacnonvanishing}\det \frac{\partial (w_{1},..,w_{n})}{\partial(z_{1},..,z_{n})}\neq 0\end{equation}
 should hold in a prescribed neighborhood of $z_{0}$. Although these conditions still fail to yield ``coordinates'' in the broad sense (since one does not have the injectivity of the mapping $z\to (w_{1},w_{2},...,w_{n})^{t}$), this information will do for the purposes of this article. 

 Quite unexpectedly it occurs that the functions $w_{i}$ are well
 defined in a geodesic ball of radius that does not depend not
 only on the choice of $z_{0}$, but is also independent of $\Omega$.
 Thus we have

\begin{theorem}\label{wheredefined}
Let $\Omega\subset\subset\co^{n}$ be a bounded domain equipped with the Bergman metric. For any $z_{0}\in \Omega$ The Bergman kernel $K(z,z_{0})$ does not vanish in the geodesic ball $\{z\in\Omega: dist_\Omega(z,z_{0})<\frac{\pi}{2}\}$. 
\end{theorem}

Here the geodesic distance is with respect to the Riemannian metric yielded by the (K\" ahlerian) Bergman metric. We note that Teorem\ref{wheredefined} is just a matter of looking from a different viewpoint at known facts.
 
 Concerning the problem of linear independency it comes out that
 \eqref{jacnonvanishing} is satisfied in a geodesic ball of radius that depends only
 on the Ricci curvature of the Bergman metric.
\begin{theorem}\label{whereinjective} Let $\Omega\subset\subset\co^{n}$ be a bounded domain equipped with the Bergman metric. Let $c\in(-\infty,n+1)$ be a global lower bound of the Ricci curvature of the Bergman metric. For any $z_{0}\in \Omega$ the mapping 
$$z\to (w_{1}(z),w_{2}(z)...,w_{n}(z))^{t}$$
is an immersion  in the geodesic ball $\{z\in\Omega: dist_\Omega(z,z_{0})<\frac{\pi}{2\sqrt{n+1-c}}\}$.
 \end{theorem}

The lower bound $c$ is defined as usually as a constant for which $Ric_{i\bar j}-cT_{i\bar j}$ is a positive definite matrix.

In the theorem above we assumed that the Ricci curvature is bounded below. This is not the case in general see \cite{dinew-2009} and \cite{zwonek-2009}.
Note that the Ricci curvature is the same as the sectional, holomorphic sectional and Gaussian curvature, when the dimension is $1$.
Moreover the example from \cite{dinew-2009}, after delicate
smoothing of the boundary, where the annuli overlap, shows that
even so strong assumptions as being bounded smooth and strictly
pseudoconvex except at a single point, which is a peak point for
any reasonable algebra of holomorphic functions, are not enough to
guarantee boundedness of the Ricci curvature.

On the other hand for $\mathcal C^{2}$ (see \cite{MR2018337}) strictly pseudoconvex domains or for domains of finite type in $\co^{2}$ (this fact is not stated explicitly in the literature, for nontangential approach see \cite{MR1398092})
one has a global lower bound for the Ricci curvature of the Bergman metric.

Nevertheless for any domain $\Omega$ we have the following local substitute for
Theorem \ref{whereinjective}.

\begin{theorem}\label{whereinjective1} Let $\Omega\subset\subset\co^{n}$ be a bounded domain equipped with the Bergman metric. Let $U\subset\Omega$ be an open set for which $\underset{\stackrel  {z\in U}{X\in\co^{n}\setminus\{0\}}}{\inf}\frac{Ric_{i\bar j}(z) X_{i}\bar X_{j}}{T_{i\bar j}(z)X_{i}\bar X_{j}}>c$.  For any $z_{0}\in \Omega$ the mapping 
$$z\to (w_{1}(z),w_{2}(z)...,w_{n}(z))^{t}$$
is an immersion  in the set  $U\cap\{z\in\Omega: dist_\Omega(z,z_{0})<\frac{\pi}{2\sqrt{n+1-c}}\}$.
 \end{theorem}

The proofs of both Theorem \ref{wheredefined} and Theorem \ref{whereinjective} are quite similar and consist of  using the
Kobayashi construction (see \cite{MR0112162}) of an imbedding  in a
infinite dimensional projective space and in the second case the
target is also the projective space, however one uses an imbedding
due to Lu Qi-Keng (see\cite{MR2447420}) in the infinite
dimensional Grassmannian and the Pl\"ucker imbedding afterwards.

It is also of interest whether the estimates in Theorem \ref{wheredefined} and Theorem \ref{whereinjective} are optimal (whether the radii of the geodesic balls are the maximal possible). From the point of view of Riemannian geometry the generic optimality of the radius in Theorem \ref{wheredefined} would mean that the Kobayashi embedding, restricted to a real submanifold of $\Omega$ is totally geodesic and that the cut-locus of $z_{0}$ lies outside $\{z\in \Omega : dist_{\Omega} (z,z_{0})<\frac{\pi}{2}\}$. Especially the first condition is very restrictive and hence one would expect that the radius estimate is not optimal. In spite of this we prove
\begin{theorem}\label{annulusestimate1}
 For any $\varepsilon>0$ there exists a domain $\Omega_{\varepsilon}$ such that there exists $z_{0}\in \Omega_{\varepsilon}$ for which $K(z,z_{0})$ has a zero in the geodesic ball $\{z\in \Omega_{\varepsilon} : dist_{\Omega_{\varepsilon}} (z,z_{0})<\frac{\pi}{2}+\varepsilon\}.$
\end{theorem}

Concerning Theorem \ref{whereinjective} the radius is not optimal since the theorem takes into account the minimum and not the actual value of the Ricci curvature. However when the right metric is assumed then a result of optimality does also hold

\begin{theorem}\label{annulusestimate2}
 For any $\varepsilon>0$ there exists a domain $\Omega_{\varepsilon}$ such that there exists $z_{0}\in \Omega_{\varepsilon}$ for which $z\to (w_{1}(z),w_{2}(z),..,w_{n}(z))^{t}$ fails to be an immersion in the whole geodesic ball $\{z\in \Omega_{\varepsilon} : \tilde {dist}_{\Omega_{\varepsilon}} (z,z_{0})<\frac{\pi}{2}+\varepsilon\}.$
\end{theorem}

{\bf Acknowledgement.}
This paper was written when the author was a Junior Fellow at the Erwin Schr\"odinger Institute in Vienna. I would like to thank the institute for hospitality and perfect working conditions. I would like also to thank professor Takeo Ohsawa for turning my attention to hyperelliptic Riemann surfaces and professor Siqi Fu for pointing out the case of domains of finite type in $\co^{2}$.
\end{section}
\begin{section}{The case of a bounded domain in $\mathbb C^{n}$}\label{section1}
In this section $\Omega$ will be a bounded domain in $\co^{n}$. Let $\varphi=\{\varphi_{0},\varphi_{1},\varphi_{2}...\}$ be an orthonormal basis of the Hilbert space $\mathcal O\cap L^{2}(\Omega)$ of square-integrable holomorphic functions on $\Omega$. The Bergman kernel of $\Omega$, $K(z,w)=K_{\Omega}(z,w)$ is defined as follows
\begin{equation}\label{bkernel}
 K(z,w):=\sum_{i=0}^{\infty} \varphi_{i}(z)\overline{\varphi_{i}(w)}.
\end{equation}
With this kernel one associates a differential $(1,1)$-form,
\begin{equation}\label{bmetric}
\sqrt{-1} \sum_{i,j=1}^{n}T_{i\bar j}(z)dz_{i}\wedge d\bar z_{j}:=\sqrt{-1}\sum_{i,j=1}^{n}\frac{\partial^{2}}{\partial z_{i}\partial \bar z_{j}}\log K(z,z)dz_{i}\wedge d\bar z_{j}
\end{equation}
In our setting this form will be everywhere positive definite and moreover one easily sees that it is a K\" ahler form with global potential. The associated metric $\sum_{i,j=1}^{n}T_{i\bar j}dz_{i}d\bar z_{j}$ is called the Bergman metric and the square of the length of a vector $X$, measured in this metric at the point $z\in\Omega$ is
\begin{equation}\label{vlength}
 \beta^{2}(z,X)=\beta^{2}_{\Omega}(z,X):=\sum_{i,j=1}^{n}T_{i\bar j}(z)X_{i}\bar X_{j},
\end{equation}
for any vector $X\in\co^{n}$. One defines the length of a  piecewise $\mathcal C^{1}$ curve $$\gamma:[0,1]\ni t\to \gamma(t)\in\Omega,$$ as
\begin{equation}\label{clength}
 \ell(\gamma):=\int_{0}^{1}\beta(t,\gamma'(t))dt
\end{equation}
and the Bergman distance between two points $z,w\in\Omega$
\begin{equation}\label{bdistance}
dist_{\Omega}(z,w):=\inf\{\ell (\gamma): \gamma\text{ is a piecewise } \mathcal C^{1} \text { curve s.t. } \gamma(0)=z,\gamma (1)=w\}.
 \end{equation}
The Bergman distance is indeed a distance and hence endows $\Omega$ with the structure of a metric space.

Further let $g(z):=\det(T_{i\bar j}(z))_{i,j=1..n}$. Recall that 
\begin{equation}\label{ricci}
\sqrt{-1}\sum_{i,j=1}^{n}Ric_{i\bar j}(z)dz_{i}\wedge d\bar z_{j}:=-\sqrt{-1}\sum_{i,j=1}^{n}\frac{\partial^{2}}{\partial z_{i}\partial \bar z_{j}}\log g(z)dz_{i}\wedge d\bar z_{j} 
\end{equation}
is the Ricci form of the Bergman metric and let
\begin{equation}\label{riccimetric}
 \sqrt{-1}\sum_{i,j=1}^{n}\tilde T_{i\bar j}(z)dz_{i}\wedge d\bar z_{j}:=\sqrt{-1}\sum_{i,j=1}^{n}\left((n+1)T_{i\bar j}(z)+\frac{\partial^{2}}{\partial z_{i}\partial \bar z_{j}}\log g(z)\right)dz_{i}\wedge d\bar z_{j}.  
\end{equation}
It follows that this form is also positive definite (see eg. \cite{MR2447420}) in our setting and K\" ahler, with K\" ahler potential $\log \left(K(z,z)^{n+1}g(z)\right)$. Slightly different construction was assumed in \cite{MR0214810}. As above one defines the square of the length of a vector $\tilde\beta^{2}(z,X)$, the length of a curve $\tilde\ell(\gamma)$ and distance $\tilde {dist}_{\Omega}(z,w)$ with respect to this new K\" ahler metric.

\begin{subsection}{The Kobayashi embedding}

The Kobayashi embedding is the holomorphic embedding of the domain $\Omega$ into the projective space over the Hilbert space dual of the Hilbert space $L^{2}\cap\mathcal O(\Omega)$ which is naturally identified with the infinite dimensional projective space $\mathbb C \mathbb P^{\infty}$. The construction goes as follows.
Fix an orthonormal basis $\varphi=\{\varphi_{0},\varphi_{1},\varphi_{2}...\}$, $\varphi_{j}\in   L^{2}\cap\mathcal O(\Omega)$.
The Kobayashi embedding is the mapping 
$\iota_{\text{\it Ko},\varphi}$ defined by
$$ \Omega\ni z\to \iota_{\text{\it Ko},\varphi}(z)=[(\varphi_{0}(z),\varphi_{1}(z),\varphi_{2}(z),...)]\in \mathbb C \mathbb P^{\infty},$$
where the above notation is with respect to the homogeneous coordinates in $\mathbb C \mathbb P^{\infty}$. One easily sees that
\begin{equation}\label{blockiprojective}
 \iota_{\text{\it Ko},\varphi}(z)=[ \langle\circ,K(.,z) \rangle_{L^{2}(\Omega)}].
\end{equation}
 
What makes this construction so important is the fact that the embedding is  isometric in the sense that the pullback $\iota_{\text{\it Ko},\varphi}^{*}\omega_{FS}$ of the standard Fubini-Study metric on $\mathbb C \mathbb P^{\infty}$ is exactly the Bergman metric of $\Omega$. This combined with the formula for the distance on the projective space gives one the following inequality (see \cite{MR2139520} or \cite{wykl} and Proposition 4.1.6 therein):
\begin{equation}\label{projectivedist}
dist_{\Omega}(z,z_{0})\geq\arccos \frac{|K(z,z_{0})|}{\sqrt{K(z,z)K(z_{0},z_{0})}}
\end{equation}
It is clear that equality need not hold in \eqref{projectivedist}.

Once one has \eqref{projectivedist}, Theorem \ref{wheredefined} follows easily, since $\arccos 0=\frac{\pi}{2}$
 \end{subsection}
\begin{subsection}{The Lu Qi-Keng embedding}

The Lu Qi-Keng embedding is in some way similar to the Kobayashi embedding however the target manifold is different.

It is defined as 

$$\Omega\ni z\to\left[\begin {pmatrix}
                        \varphi_{0}\frac{\partial\varphi_{1}}{\partial z_{1}}
-\varphi_{1}\frac{\partial\varphi_{0}}{\partial z_{1}} & \varphi_{0}\frac{\partial\varphi_{2}}{\partial z_{1}}-\varphi_{2}\frac{\partial\varphi_{0}}{\partial z_{1}}&\varphi_{1}\frac{\partial\varphi_{2}}{\partial z_{1}}-\varphi_{2}\frac{\partial\varphi_{1}}{\partial z_{1}} & \cdots\\
\vdots&\vdots&\vdots &\cdots\\
  \varphi_{0}\frac{\partial\varphi_{1}}{\partial z_{n}}
-\varphi_{1}\frac{\partial\varphi_{0}}{\partial z_{n}} &  \varphi_{0}\frac{\partial\varphi_{2}}{\partial z_{n}}-\varphi_{2}\frac{\partial\varphi_{0}}{\partial z_{n}}&\varphi_{1}\frac{\partial\varphi_{2}}{\partial z_{n}}-\varphi_{2}\frac{\partial\varphi_{1}}{\partial z_{n}} & \cdots                   \end {pmatrix}
\right]\in \mathbb F(n,\infty),$$
the infinite-dimensional Grassmanian of $n$-dimensional planes. It is implicitly assumed that if
\begin{equation}\label {assumptions}
             p=\left[\begin {pmatrix}
                        p_{11}  &p_{12} &\cdots \\
\vdots &\vdots&\cdots\\
p_{n1}  &p_{n2}      &\cdots           \end {pmatrix}
\right]\in \mathbb F(n,\infty)\text{, then } \sum_{j=1}^{\infty}|p_{ij}|^2<\infty, i=1..n.\end{equation}
It is proved in \cite{MR2447420} that the pullback of the Fubini-Study metric of the Grassmanian ($\tilde\omega_{FS}$) which can be seen as the metric associated to the form
$$Tr ((I_{n}+ZZ^{*})^{-1}dZ\wedge(I+Z^{*}Z)^{-1}dZ^{*}),$$ 
in local coordinates $Z$ of the Grassmannian is
\begin{equation}\label{luqi} \iota_{\text{\it Lu},\varphi}^{*}\tilde\omega_{FS}=(n+1) T_{i\bar j}-Ric_{i\bar j},\end{equation} 
where $Ric $ is the Ricci tensor of the Bergman metric (and hence coincides with $\tilde\beta^{2}(.,\circ)$).

One has that $\tilde\omega_{FS}$ is itself the pullback via the Pl\" ucker embedding of $\omega_{FS}$
in $\mathbb C \mathbb P^{\infty}$,
\begin{equation}\label{plucker}\tilde\omega_{FS}=\iota_{\text{\it Pl\"u}}^{*}\omega_{FS}.\end{equation} 

This is almost immediate generalization of the finite-dimensional case, however the author was unable to find this result in the literature and hence a proof is provided below.
For the finite-dimensional Grassmannian this is done in \cite{MR0126808}, Satz $7$.

Without loss of generality (by a transitivity argument) one can assume that $p\in \mathbb F(n,\infty)$ lies in the subset of $\mathbb F(n,\infty)$ for which the matrix representing $p$,

$$\begin {pmatrix}
                        p_{11}  &p_{12} &\cdots \\
\vdots &\vdots&\cdots\\
p_{n1}  &p_{n2}      &\cdots           \end {pmatrix}
\text{, for } p=\left[\begin {pmatrix}
                        p_{11}  &p_{12} &\cdots \\
\vdots &\vdots&\cdots\\
p_{n1}  &p_{n2}      &\cdots           \end {pmatrix}
\right],$$
has the property that exactly the first $n\times n$ minor,
$\begin {pmatrix}
                        p_{11}\cdots p_{1n} \\
\vdots\ddots\vdots\\
p_{n1}\cdots p_{nn}                 \end {pmatrix}
$ is nonsingular. In fact every matrix representing the class $p$ will have the required property.
Let
$$Z=\begin {pmatrix}
                        z_{11}  &z_{12} &\cdots \\
\vdots &\vdots&\cdots\\
z_{n1}  &z_{n2}      &\cdots           \end {pmatrix}
:=\begin {pmatrix}
                        p_{11}\cdots p_{1n} \\
\vdots\ddots\vdots\\
p_{n1}\cdots p_{nn}                 \end {pmatrix}
^{-1}\begin {pmatrix}
                        p_{1 n+1}  &p_{1 n+2} &\cdots \\
\vdots &\vdots&\cdots\\
p_{n n+1}  &p_{n n+2}      &\cdots           \end {pmatrix}.$$
Then the matrix obtained by pairing the blocks $(I_{n}, Z)$ represents $p$, moreover a representative of this type is unique. One says that $Z$ is the local coordinate of $p$ in the neighbourhood of $\left[\begin {pmatrix}
                        1& \cdots & 0 &0&\cdots \\
\vdots &\ddots&\vdots&\vdots& \cdots\\
0& \cdots &1  &   0 &\cdots           \end {pmatrix}
\right].$

The Pl\" ucker embedding sends the vector space  spanned by the vectors $v_{1},...,v_{n}\in \co ^{\infty}$ into the element $[v_{1}\wedge\cdots \wedge v_{n}]\in P(\Lambda^{n}\co^{\infty})\cong \mathbb C\mathbb P^{\infty} $, where $P(\Lambda^{n}\co^{\infty})$ is the projectivization of $\Lambda^{n}\co^{\infty}$.

In local coordinates this reads

$$\left[\begin {pmatrix}
                        1& \cdots & 0 &z_{11} & z_{12} &\cdots \\
\vdots &\ddots&\vdots&\vdots&\vdots&\cdots\\
0& \cdots &1  &   z_{n1}&z_{n2} &\cdots           \end {pmatrix}
\right]_{\mathbb F(n,\infty)}\to[(e_{1}+z_{11}e_{n+1}+z_{12}e_{n+2}+\cdots)\wedge$$$$\wedge(e_{2}+z_{21}e_{n+1}+z_{22}e_{n+2}+\cdots)\wedge\cdots\wedge(e_{n}+z_{n1}e_{n+1}+z_{n2}e_{n+2}+\cdots)]_{P(\Lambda^{n}\co^{\infty})}=$$
$$\left[e_{1}\wedge\cdots\wedge e_{n}+\sideset{}{'}{\sum}_{\stackrel{(j_{1},j_{2},...,j_{n})\neq}{(-n+1,-(n-1)+1,...,0)}}\det \begin {pmatrix}
                        z_{1j_{1}}\cdots z_{1j_{n}} \\
\vdots\ddots\vdots\\
z_{nj_{1}}\cdots z_{nj_{n}}                 \end {pmatrix}
 e_{j_{1}+n}\wedge\cdots\wedge e_{j_{n}+n}\right],$$

where we assume $\begin {pmatrix}
                        z_{1 -n}\cdots z_{10} \\
\vdots\ddots\vdots\\
z_{n -n}\cdots z_{n 0}                 \end {pmatrix}=I_{n}$

The isomorphism of $P(\Lambda^{n}\co^{\infty})$ with $\mathbb C\mathbb P^{\infty}$ is realized by enumerating lexicographically 
$$\tilde e_{s}= e_{j_{1}(s)+n}\wedge e_{j_{2}(s)+n}\wedge\cdots\wedge e_{j_{n}(s)+n}.$$
Because $\tilde e_{0}=e_{1}\wedge...\wedge e_{n}$ the local coordinate  of the image of $p$ in $\mathbb C\mathbb P^{\infty}$ will be
$$\sum_{s=1}^{\infty}\det \begin {pmatrix}
                        z_{1j_{1}(s)}\cdots z_{1j_{n}(s)} \\
\vdots\ddots\vdots\\
z_{nj_{1}(s)}\cdots z_{nj_{n}(s)}                 \end {pmatrix}
 \tilde e_{s}.$$

The Fubini-Study metric $\omega_{FS}$ is the metric associated to $\partial\bar\partial\log (1+WW^{*})$, for $W=(w_{1},w_{2},...)$, the local coordinate of $w=[(1,w_{1},w_{2},...)]$, the line with direction  $\tilde e_{0}+\sum_{s=1}^{\infty}w_{s}\tilde e_{s}$ in $\co^{\infty}$.

The metric at the image point of $(I_{n}, Z)$ is associated to
$$\partial\bar\partial \log\left(1+\sum_{s=1}^{\infty}\left|\det\begin {pmatrix}
                        z_{1j_{1}(s)}\cdots z_{1j_{n}(s)} \\
\vdots\ddots\vdots\\
z_{nj_{1}(s)}\cdots z_{nj_{n}(s)}                 \end {pmatrix}
 \right|^{2}\right),$$
which by the Cauchy-Binet formula equals $\partial\bar\partial \log\det (I+ZZ^{*})$
(there is no problem with convergence here, by the assumption \eqref{assumptions}).

We use the well known expressions for the derivative of the determinant and the inverse matrix (all the notations are to be understood in the obvious sense).

$$\bar \partial \det A=\det A \text{\it Tr} (A^{-1}\bar\partial A) $$
$$\partial A^{-1}=-A^{-1}(\partial A)A^{-1}$$

$$\partial\bar\partial \log(\det (I+ZZ^{*}))=\partial\frac{det (I+ZZ^{*})Tr(\ZZ^{-1}\bar\partial\ZZ)}{det (I+ZZ^{*})}=$$
$$\partial Tr(\ZZ^{-1}ZdZ^{*})=Tr(-\ZZ^{-1}(\partial\ZZ)\ZZ^{-1}ZdZ^{*}+$$$$+\ZZ^{-1}dZdZ^{*})=Tr(\ZZ^{-1}dZ(I-Z^{*}\ZZ^{-1} Z)dZ^{*})).$$
What remains is to show that 
$$I-Z^{*}\ZZ^{-1} Z=(I+Z^{*}Z)^{-1}$$

Multiplying with $I+Z^{*}Z$ gives one

$$(I-Z^{*}\ZZ^{-1} Z)(I+Z^{*}Z)=I+Z^{*}Z-Z^{*}\ZZ^{-1}(Z+ZZ^{*}Z)=$$$$I+Z^{*}Z -Z^{*}\ZZ^{-1}\ZZ Z= I+Z^{*}Z-Z^{*}Z=I,$$
hence
$$Tr (I_{n}+ZZ^{*})^{-1}dZ\wedge(I+Z^{*}Z)^{-1}dZ^{*}=\partial\bar\partial \log\det (I+ZZ^{*}),$$
which proves \eqref{plucker}.

Now combining \eqref{plucker} and \eqref{luqi} one has
\begin{equation}\label{combinpullback}  (n+1) T_{i\bar j}-Ric_{i\bar j}=\iota_{\text{\it Lu},\varphi}^{*}\iota_{\text{\it Pl\"u}}^{*}\omega_{FS}. 
\end{equation}
And hence $ \iota_{\text{\it Pl\"u}}\circ \iota_{\text{\it Lu},\varphi}$ is an isometric embedding of $\Omega$ with the metric $\tilde \beta$ to $\mathbb C \mathbb P^{\infty}$ with the Fubini-Study metric. Now using the formula for the geodesic distance in $\mathbb C \mathbb P^{\infty}$ one obtains, like \eqref{projectivedist},

$$\tilde {dist} (z,\zeta)\geq $$$$\arccos\frac{|(\iota_{\text{\it Pl\"u}}\circ \iota_{\text{\it Lu},\varphi}(z))(\iota_{\text{\it Pl\"u}}\circ \iota_{\text{\it Lu},\varphi}(\zeta))^{*}|}{\sqrt{(\iota_{\text{\it Pl\"u}}\circ \iota_{\text{\it Lu},\varphi}(z))(\iota_{\text{\it Pl\"u}}\circ \iota_{\text{\it Lu},\varphi}(z))^{*}(\iota_{\text{\it Pl\"u}}\circ \iota_{\text{\it Lu},\varphi}(\zeta))(\iota_{\text{\it Pl\"u}}\circ \iota_{\text{\it Lu},\varphi}(\zeta))^{*} }}=$$
\begin{equation}\label{arccos}\arccos\frac{\left|\det \left(K(z,\zeta)^{2}\frac{\partial^{2}}{\partial z_{i}\partial \bar \zeta_{j}}\log K(z,\zeta)\right)_{i,j=1..n}\right|}{\sqrt{ \det (K(z,z)^2T_{i\bar j}(z))_{i,j=1..n}  \det (K(\zeta,\zeta)^{2}T_{i\bar j}(\zeta))_{i,j=1..n}}}.\end{equation}

In the expression for $w_{i}$ the term $T^{\bar j i}(z_{0})$ is introduced for the sake of normalization and is irrelevant when it comes to linear independency. Hence
$$z\to (w_{1}(z),w_{2}(z)...,w_{n}(z))^{t}$$
is an immerison exactly when 
$$z\to \left(\frac{\partial}{\partial \bar{\zeta_{1}}}\log {\frac{K(z,\zeta)}{K(\zeta,\zeta)}}_{|\zeta=z_{0}},\frac{\partial}{\partial \bar{\zeta_{2}}}\log {\frac{K(z,\zeta)}{K(\zeta,\zeta)}}_{|\zeta=z_{0}},...,\frac{\partial}{\partial \bar{\zeta_{n}}}\log {\frac{K(z,\zeta)}{K(\zeta,\zeta)}}_{|\zeta=z_{0}}\right)^{t}$$ 
is an immersion. The determinant of the Jacobian of the latter expression is 
\begin{equation}\label{wherebad}\det\left(\frac{\partial}{\partial z_{i}}\frac{\partial}{\partial \bar{\zeta_{j}}}\log {\frac{K(z,\zeta)}{K(\zeta,\zeta)}}_{|\zeta=z_{0}}\right)_{i,j=1..n}=\det \left(\frac{\partial^{2}}{\partial z_{j}\partial \bar{\zeta_{j}}}\log K(z,\zeta)_{|\zeta=z_{0}}\right)_{i,j=1..n}.\end{equation}

Comparing \eqref{arccos} and \eqref{wherebad} one notices that the zero-sets of the determinants are the same, with possible difference of the singular locus. Hence for fixed $z_{0}$ the nearest point $z$ for which $\frac{\partial (w_{1}(z),..,w_{n}(z))}{\partial(z_{1},..,z_{n})}=0$ must satisfy 
\begin{equation}\label{finalll}\tilde {dist} (z,z_{0})\geq \arccos 0= \frac{\pi}{2}.\end{equation}
 Theorems \ref{whereinjective} and \ref{whereinjective1} follow.

We note that the same conclusion can be obtained by directly calculating the distance on the Grassmanian, however this is technically involved, see \cite{MR1443974} for a sketch in the finite-dimensional case.
\begin{remark}\label{w1subsetw2} Let $W_{z_{0}}$ denote the set $\{z\in\Omega: K(z,z_{0})=0\}$ and $\tilde W_{z_{0}}$ denote the set $\{z\in\Omega: \det \left(K(z,z_{0})^{2}\frac{\partial^{2}}{\partial z_{i}\partial \bar \zeta_{j}}\log K(z,\zeta)_{|\zeta=z_{0}}\right)_{i,j=1..n}=0\}$. If $n>1$ one has that $W_{z_{0}}\subset \tilde W_{z_{0}}$.
\end{remark}
\begin{proof} It follows by a simple calculation that
 $$\det \left(K(z,z_{0})^{2}\frac{\partial^{2}}{\partial z_{i}\partial \bar \zeta_{j}}\log K(z,\zeta)_{|\zeta=z_{0}}\right)_{i,j=1..n}=$$$$\det \left(K(z,z_{0})\frac{\partial^{2}}{\partial z_{i}\partial \bar \zeta_{j}} K(z,\zeta)_{|\zeta=z_{0}}-\frac{\partial}{\partial z_{i}} K(z,z_{0})\frac{\partial}{\partial \bar \zeta_{j}} K(z,\zeta)_{|\zeta=z_{0}}\right)_{i,j=1..n}.$$

When $z\in W_{z_{0}}$ this reduces to 
$$\det \left(-\frac{\partial}{\partial z_{i}} K(z,z_{0})\frac{\partial}{\partial \bar \zeta_{j}} K(z,\zeta)_{|\zeta=z_{0}}\right)_{i,j=1..n}=0,$$
since the matrix is of rank 1.
\end{proof}

\end{subsection}
 \end{section}
\begin{section}{The manifold case}

The essential difference between the manifold and the domain cases is that there do not exist coordinates in the large. Moreover in the compact case there are no nonconstant holomorphic functions. Therefore one has to modify the construction of the Bergman kernel and to employ forms of top degree instead of functions.

Let $M$ be a $n$- dimensional complex manifold. The space of top degree holomorphic forms is denoted by $H^{0}(M,K_{M})$, which can also be viewed as the space of global holomorphic sections of the canonical bundle $K_{M}$ over $M$. 
One can restrict to the space $H^{0}_{(2)}(M,K_{M})=H^{0}(M,K_{M})\cap L^{2}(M,K_{M})$ of square-integrable global holomorphic forms of top degree, i.e.,
$$H^{0}_{(2)}(M,K_{M})=\{f\in H^{0}(M,K_{M}): \sqrt{-1}^{n^2}\int_{M}f\wedge\bar f<\infty\}.$$

Now $H^{0}_{(2)}(M,K_{M})$
can be equipped with an inner product 
$$H^{0}_{(2)}(M,K_{M})\ni f,g\to \sqrt{-1}^{n^2}\int_{M}f\wedge\bar g\in \co.$$
This inner product turns $H^{0}_{(2)}(M,K_{M})$ into a (possibly finite dimensional) Hilbert space. Note that if $M$ is compact $H^{0}_{(2)}(M,K_{M})\equiv H^{0}(M,K_{M})$ and $ H^{0}_{(2)}(M,K_{M})$ is finitely-dimensional.
 Let $\varphi=\{\varphi_{0},\varphi_{1},\varphi_{2}...\}$ be an orthonormal basis of $H^{0}_{(2)}(M,K_{M})$. The Bergman kernel form is 
$$K=\sum_{j>0}\varphi_{j}\wedge\bar\varphi_{j}$$ In local coordinates one can write.
$$K(z,\zeta)=K^{*}(z,\zeta)dz_{1}\wedge dz_{2}\wedge..\wedge dz_{n}\wedge d\bar \zeta_{1}\wedge \bar \zeta_{2}\wedge..\wedge d\bar \zeta_{n},$$
where $K^{*}(z,\zeta)$ is a function which is defined only locally.  
The $(1,1)$- differential form 
 \begin{equation}\label{manmetric}\partial\bar\partial \log K^{*}(z,z)\end{equation}
is however globally defined, which can be easily seen by expressing $K(z,\zeta)$ in different local coordinates. 

 One says that that the manifold $M$ has a Bergman metric if the form \eqref{manmetric} is globally strictly positive. In such a case the Bergman metric is the metric associated to \eqref{manmetric}. One can therefore define $T_{i\bar j}(z)$, $\beta(z,X)$, $dist_{M} (z,w)$,  $Ric_{i\bar j}$, $\tilde\beta(z,X)$,  and $\tilde{dist}_{M} (z,w)$ like in the previous section with the constraint that $K^{*}$ is defined only locally.

Suppose now that $M$ carries a Bergman metric. Having  the starred counterparts of functions in $\co^{n}$ one sees that the representative coordinates can also be defined at least locally. We take an open cover $\{U_{i}\}$ of $M$ subordinate to local coordinate charts. Let $z_{0}\in U_{1}$ be fixed. After probably shrinking we can arrange the sets $U_{i}\times U_{1}$ to cover $M\times U_{1}$ and in every $ U_{i}\times U_{1}$, $K(z,\zeta)$ can be expressed by
$$K(z,\zeta)=K^{*}_{U_{i}}(z,\zeta) d z_{1}\wedge..\wedge d z_{n}\wedge d\bar\zeta_{1}\wedge..\wedge d\bar\zeta_{n}.$$ In $U_{i}\times U_{1}\cap U_{j}\times U_{1}$ the change of coordinates $(z,\zeta)\stackrel{(\tilde z(.),id)}{\longrightarrow} (\tilde z(z),\zeta)$ yields 
$$K^{*}_{U_{j}}(\tilde z(z),\zeta) \frac{\partial (\tilde z(z)_{1},..,\tilde z(z)_{n})}{\partial(z_{1},..,z_{n})}=K^{*}_{U_{i}}( z,\zeta)$$

For every $s\in\{1..n\}$ one has
$$\frac{\partial}{\partial \bar \zeta_{s}}\log K^{*}_{U_{i}}( z,\zeta)=\frac{\partial}{\partial \bar \zeta_{s}}\log \left (K^{*}_{U_{j}}(\tilde z(z),\zeta) \frac{\partial (\tilde z(z)_{1},..,\tilde z(z)_{n})}{\partial(z_{1},..,z_{n})}\right)=\frac{\partial}{\partial \bar \zeta_{s}}\log K^{*}_{U_{j}}( \tilde z,\zeta).$$

Hence the expressions 
$$w_{l}(z)=\sum_{k=1}^{n}T^{\overline{k}l}(z_{0})\frac{\partial}{\partial \bar{\zeta_{k}}}\log {\frac{K^{*}_{U_{i}}(z,\zeta)}{K^{*}_{U_{1}}(\zeta,\zeta)}}_{|\zeta=z_{0}},l=1..n$$
glue up to  global functions. As in the $\co ^{n}$ case the only obstruction that can appear is that $K_{U_{i}}(z,z_{0})$ may be zero for some $z$ (this is clearly independent on the set $U_{i}$, the representation will be zero for every $U_{j}$, $z\in U_{j}$). 

We remark that representative coordinates for the Bergman metric on manifolds were previously studied in \cite{MR0466638}, however there the Bergman kernel function instead of the Bergman kernel form was used. This substantially limited the range of assumed manifolds, for example every compact manifold was excluded from consideration.

From now on the convention will be that $f^{*}$ is the local coefficient of the form $f(z)=f^{*}(z)dz_{1}\wedge dz_{2}\wedge..\wedge dz_{n}$.

Unlike the situation in $\co^{n}$, $M$ does not obviously possess Bergman metric. The necessary and sufficient conditions for $M$ to have a Bergman metric are the following:

 $\bullet$  For every $z\in M$ there exists $f\in H^{0}_{(2)}(M,K_{M})$ such that $f^{*}(z)\neq 0$ (condition A.1 in \cite{MR0112162}).

  $\bullet$ For every $z\in M$ and for every vector $X$ in the complex tangent space at $z$ there exists $g\in H^{0}_{(2)}(M,K_{M})$ such that $g^{*}(z)=0$ and $X(g^{*})(z)\neq 0$ (condition A.2 in \cite{MR0112162}).

The later condition is clearly equivalent to 

 For every $z\in M$ and for a basis  $X_{i},i=1..n$ of the complex tangent space at $z$ there exist $g_{i}\in H^{0}_{(2)}(M,K_{M})$ such that $g_{i}(z)=0$ and $X_{i}(g_{i}^{*})(z)\neq 0$.

These conditions are hard to check for an abstract complex manifold, however one immediately sees that a necessary condition for $M$ to have a Bergman metric is that the $(n,0)$- Hodge number $h^{n,0}(M)$ (or geometric genus) satisfies $h^{n,0}(M)\geq n+1$.

It turns out (see \cite{MR0112162}) that A.1 and A.2 are also necessary and sufficient conditions for $\iota_{\text{\it Ko},\varphi}$ to be an immersion (A.1 solely is necessary and sufficient for the Kobayashi mapping to be well defined), whereas for an injection one needs another condition (A.3 in \cite{MR0112162}):

$\bullet$ For every two points $z, z_{0}\in M$ there exists $h\in H^{0}_{(2)}(M,K_{M})$ such that  $h(z)=0$ and $h(z_{0})\neq 0$.

In order to carry the Lu Qi-Keng construction on manifolds, one first has to check that $\iota_{\text{\it Lu},\varphi}$  does not depend on local holomorphic coordinate changes (this is not completely obvious, since there are partial derivatives in the expression for $\iota_{\text{\it Lu},\varphi}$).

Let
$$\varphi_{j}=\varphi_{j}^{*}(\tilde z) d \tilde z_{1}\wedge..\wedge d\tilde z_{n}=\varphi_{j}^{*}(\tilde z (z))Jac\left(\frac{\partial \tilde z}{\partial z}\right) d  z_{1}\wedge..\wedge d z_{n}$$

Now

$$ \varphi^{*}_{j}( z)\frac{\partial\varphi^{*}_{k}( z)}{\partial z_{s}}
-\varphi^{*}_{k}( z)\frac{\partial\varphi^{*}_{j}(
z)}{\partial z_{s}}=$$

$$ \tilde\varphi^{*}_{j}(\tilde
z(z))Jac\left(\frac{\partial \tilde z}{\partial
z}\right)\frac{\partial\tilde\varphi^{*}_{k}(\tilde
z(z))Jac\left(\frac{\partial \tilde z}{\partial
z}\right)}{\partial z_{s}} -\tilde\varphi^{*}_{k}(\tilde
z(z))Jac\left(\frac{\partial \tilde z}{\partial
z}\right)\frac{\partial\tilde\varphi^{*}_{j}(\tilde
z(z))Jac\left(\frac{\partial \tilde z}{\partial
z}\right)}{\partial z_{s}}=$$

$$Jac\left(\frac{\partial \tilde z}{\partial z}\right)^{2}\tilde\varphi^{*}_{j}(\tilde z(z))\sum_{m=1}^{n} {\frac{\partial\tilde\varphi^{*}_{k}}{\partial \tilde z_{m}}}\frac{\partial \tilde z_{m}}{\partial z_{s}} -Jac\left(\frac{\partial \tilde z}{\partial z}\right)^{2}\tilde\varphi^{*}_{k}(\tilde z(z))\sum_{m=1}^{n} \frac{\partial\tilde\varphi^{*}_{j}}{\partial \tilde z_{m}}\frac{\partial \tilde z_{m}}{\partial z_{s}}.$$
Hence

$$\left(\varphi^{*}_{j}( z)\frac{\partial\varphi^{*}_{k}(z)}{\partial z_{s}}
-\varphi^{*}_{k}(z)\frac{\partial\varphi^{*}_{j}(
z)}{\partial z_{s}}\right)
_{\stackrel{s=1..n}{j<k}}=$$
$$Jac\left(\frac{\partial
\tilde z}{\partial z}\right)^{2}\begin{pmatrix}
                                                                                                                                                     \frac{\partial \tilde z_{1}}{\partial z_{1}}&\cdots&  \frac{\partial \tilde z_{n}}{\partial z_{1}}\\
         \vdots&\ddots&\vdots\\
     \frac{\partial \tilde z_{1}}{\partial z_{n}}&\cdots&  \frac{\partial \tilde z_{n}}{\partial z_{n}}                                                                                                                                      \end{pmatrix} \left(\tilde\varphi^{*}_{j}(\tilde z)\frac{\partial\tilde\varphi^{*}_{k}(\tilde z)}{\partial\tilde z_{s}}
-\tilde\varphi^{*}_{k}(\tilde z)\frac{\partial\tilde\varphi^{*}_{j}(\tilde
z)}{\partial\tilde
z_{s}}\right)_{\stackrel{s=1..n}{j<k}}$$
and the classes of these matrices coincide since $Jac\left(\frac{\partial \tilde z}{\partial z}\right)^{2}\begin{pmatrix}
                                                                                                                                                     \frac{\partial \tilde z_{1}}{\partial z_{1}}&\cdots&  \frac{\partial \tilde z_{n}}{\partial z_{1}}\\
         \vdots&\ddots&\vdots\\
     \frac{\partial \tilde z_{n}}{\partial z_{1}}&\cdots&  \frac{\partial \tilde z_{n}}{\partial z_{n}}                                                                                                                                      \end{pmatrix}$ is nonsingular.

A careful analysis of \cite{MR2447420} gives that the necessary and sufficient conditions for $\iota_{\text{\it Lu},\varphi}$ to be well defined are

 $\bullet$  For every $z\in M$ there exists $f\in H^{0}_{(2)}(M,K_{M})$ such that $f^{*}(z)\neq 0$ (the same as condition A.1).

 $\bullet$ For every $z\in M$ there exist $f_{1},..,f_{n}\in H^{0}_{(2)}(M,K_{M})$ such that $f^{*}_{i}(z)= 0,i=1..n$  and
$\begin{pmatrix}\frac{\partial f_{1}^{*}}{\partial z_{1}}& \cdots&\frac{\partial f_{n}^{*}}{\partial z_{1}}\\
\vdots&\ddots&\vdots\\
  \frac{\partial f_{1}^{*}}{\partial z_{n}} &\cdots& \frac{\partial f_{n}^{*}}{\partial z_{n}}
 \end{pmatrix}
$ is non-singular at $z$ (condition B.1)

The necessary and sufficient condition for $\iota_{\text{\it Lu},\varphi}$ to be an immersion, in addition to A.1 and B.1, is

$\bullet$ For every $z\in M$ there exist $g_{1},..,g_{\frac{n(n+1)}{2}}\in H^{0}_{(2)}(M,K_{M})$ (some of them probably $0$) such that $g^{*}_{i}(z)= 0,i=1..\frac{n(n+1)}{2}$,  $dg^{*}_{i}= 0,i=1..\frac{n(n+1)}{2}$ at the point $z$ and the $n\times \frac{n^{2}(n+1)}{2}$ matrix 
$$\begin{pmatrix}\frac{\partial^{2} g_{1}^{*}}{\partial z_{1}\partial z_{1}}& \cdots&\frac{\partial^{2} g_{1}^{*}}{\partial z_{n}\partial z_{1}}&\frac{\partial^{2} g_{2}^{*}}{\partial z_{1}\partial z_{1}}& \cdots&\frac{\partial^{2} g_{2}^{*}}{\partial z_{n}\partial z_{1}}& \cdots & \frac{\partial^{2} g_{\frac{n(n+1)}{2}}^{*}}{\partial z_{1}\partial z_{1}}& \cdots&\frac{\partial^{2} g_{\frac{n(n+1)}{2}}^{*}}{\partial z_{n}\partial z_{1}}\\
\vdots&\ddots&\vdots&\vdots&\ddots&\vdots&\cdots&\vdots&\ddots&\vdots\\
  \frac{\partial^{2} g_{1}^{*}}{\partial z_{1}\partial z_{n}}& \cdots&\frac{\partial^{2} g_{1}^{*}}{\partial z_{n}\partial z_{n}}&\frac{\partial^{2} g_{2}^{*}}{\partial z_{1}\partial z_{n}}& \cdots&\frac{\partial^{2} g_{2}^{*}}{\partial z_{n}\partial z_{n}}& \cdots & \frac{\partial^{2} g_{\frac{n(n+1)}{2}}^{*}}{\partial z_{1}\partial z_{n}}& \cdots&\frac{\partial^{2} g_{\frac{n(n+1)}{2}}^{*}}{\partial z_{n}\partial z_{n}}
 \end{pmatrix}
$$ has rank $n$ at $z$ (condition B.2)

Finally the necessary and sufficient condition for $\iota_{\text{\it Lu},\varphi}$ to be an injection is

$\bullet$ For every pair of distinct points $z,w\in M$ and for every nonsingular $n\times n$ matrix $(p_{ij})_{i,j=1..n}$ there exist $f,g\in H^{0}_{(2)}(M,K_{M})$ such that 

$$\left(f^{*}(z)\frac{\partial g^{*}}{\partial z_{i}}(z)-g^{*}(z)\frac{\partial f^{*}}{\partial z_{i}}(z)\right)^{t}\neq (p_{ij})_{i,j=1..n}\left(f^{*}(w)\frac{\partial g^{*}}{\partial z_{j}}(w)-g^{*}(w)\frac{\partial f^{*}}{\partial z_{j}}(w)\right)^{t},$$
as vectors i.e., the vector with $i$-th component - the left hand side is not equal to the vector with $i$-th component- the expression on the right hand side. For some $i$ the components may however be equal. (condition B.3)

\begin{proposition}\label{injinj} Let $M$ be a complex manifold of dimension $n>1$ for which the Lu Qi-Keng mapping is well defined and the Kobayashi mapping is an injection. Then the Lu Qi-Keng mapping is also an injection.
\end{proposition}

\begin{proof} Fix the points $z,w\in M$ and the matrix $(p_{ij})_{i,j=1..n}$.
 It is enough to find $f$ and $g$ satisfying B.3. Consider the condition B.1 at $z$. There must be a non-singular $2\times 2$ minor $\begin{pmatrix}\frac{\partial f_{p}^{*}}{\partial z_{r}}& \frac{\partial f_{q}^{*}}{\partial z_{r}}\\
  \frac{\partial f_{p}^{*}}{\partial z_{s}} & \frac{\partial f_{q}^{*}}{\partial z_{s}}
 \end{pmatrix}$, $p\neq q$, $r\neq s$ of the matrix in condition B.1. One can find (complex) constants $A,B$ such that $A\frac{\partial f_{p}^{*}}{\partial z_{r}}(z)+ B\frac{\partial f_{q}^{*}}{\partial z_{r}}(z)=0$ and $A\frac{\partial f_{p}^{*}}{\partial z_{s}}(z)+B \frac{\partial f_{q}^{*}}{\partial z_{s}}(z)=1$. Let $f=Af_{p}+Bf_{q}$. It is clear that $f^{*}(z)=0$ and $\frac{\partial f^{*}}{\partial z_{r}}(z)=0,\frac{\partial f^{*}}{\partial z_{s}}(z)=1$. For $f^{*}(w)$ there are two possibilities.

In case $f^{*}(w)\neq 0$ one can find (by condition A.2, following from B.1) $g\in H^{0}_{(2)}(M,K_{M})$ such that $g^{*}(w)=0$  and $X(g^{*})(w)\neq 0$ where $X=(p_{r1},p_{r2},...,p_{rn})^{t}$. The left hand side of the expression in condition B.3 is $0$ for $i=r$ and the left hand side becomes $$\sum_{j=1}^{n} p_{rj}f^{*}(w)\frac{\partial g^{*}}{\partial w_{j}}(w)=f^{*}(w)X(g^{*})(w)\neq 0.$$

 In case $f^{*}(w)= 0$ one can find (by condition A.3) $g\in H^{0}(M,K_{M})$ such that $g^{*}(w)=0$ and $g^{*}(z)\neq 0$. So $$f^{*}(z)\frac{\partial g^{*}}{\partial z_{s}}(z)-g^{*}(z)\frac{\partial f^{*}}{\partial z_{s}}(z)=-g^{*}(z)\neq 0.$$ And
$$\sum_{j=1}^{n} p_{sj}\left(f^{*}(w)\frac{\partial g^{*}}{\partial z_{j}}(w)-g^{*}(w)\frac{\partial f^{*}}{\partial z_{j}}(w)\right)=\sum_{j=1}^{n} p_{sj}0=0.$$
\end{proof}

The simplest example of a manifold for which the Kobayashi mapping is an immersion almost everywhere and not allowing the Lu Qi-Keng mapping is a compact Riemann surface of genus $2$ (That the Kobayashi construction is an immersion outside the Weierstrass points follows by \cite{MR0237769}. Note that there ``Bergman metric'' is different from our notion of Bergman metric. On the other hand every compact Riemann surface of genus $2$ is necessarily hyperelliptic and hence the Kobayashi mapping is not a global immersion). The generic non-hyperelliptic compact Riemann surface of genus $3$ is an example of a manifold for which the Kobayashi mapping is a global immersion, however the Lu Qi-Keng mapping fails to be an immersion exactly at the $24$ Weierstrass points.

Now we see that Theorems \ref{wheredefined},\ref{whereinjective} and \ref{whereinjective1} hold also for complex manifolds under the assumption that the Kobayashi (respectively Lu Qi-Keng) mapping is an immersion.
\begin{proposition}\label{diameterofcompact} Let $M$ be a compact complex manifold admitting the Bergman metric. Then $diam M\geq \frac{\pi}{2}$ where the diameter is taken with respect to the Bergman metric.
\end{proposition}
\begin{proof}
 Fix a point $z_{0}\in M$ if $diam M <\frac{\pi}{2}$ then $M$ is contained in the geodesic ball $\{z\in M: dist_{M} (z,z_{0})<\frac{\pi}{2}\}$ and hence by Theorem \ref{wheredefined} $w_{1}$ is a globally defined nonconstant holomorphic function which clearly can not exist.
\end{proof}

\end{section}

\begin{section}{ Examples} It was Skwarczy\'nski (see\cite{MR0244512}) that first observed that for some domains $K(z,\zeta)$ has zeros, namely he proved that this is the case for the circular annulus $$P_{r}:=\{z\in\co :r<|z|<1\},r<e^{-2}.$$ Later Rosenthal (see \cite{MR0239066}) extended this result for all nondegenerate annuli by using different method. Although technically complicated the case of a planar annulus is still the easiest to study. What follows is essentially a more detailed study of the analysis in \cite{MR0244512}. Recall that the Bergman kernel of $P_{r}$ is
\begin{equation}\label{kernel} K(z,\zeta)=-\frac{1}{\pi  z\overline{\zeta} \log(r^2)}+\pi^{-1} \sum
_{j=0}^{\infty } \left(\frac{r^{2+2 j}}{\left(-r^{2+2
j}+z\overline{\zeta}\right)^2}+\frac{r^{2 j}}{\left(1-r^{2 j}
z\overline{\zeta}\right)^2}\right).\end{equation}

 Fix a positive
$\varepsilon<<1$. From now on we restrict the range of $r$ to the values for which
all the following three inequalities hold simultaneously
\begin{equation}\label{eq1}\left|\frac{1}{\log(r^2)}\right|<\varepsilon^2,\end{equation}
\begin{equation}\label{eq2}\left|r\log(r^2)\right|<\varepsilon,\end{equation}
\begin{equation}\label{eq3}
\frac{r^2}{1-r^{2}}<\varepsilon^{2}.\end{equation} It is easy to
see that all these are satisfied by all sufficiently small $r$.

 For the special choice
$$z=\frac{1}{\sqrt{|\log(r^2)|}},\zeta=\frac{-1}{(1+\varepsilon)\sqrt{|\log(r^2)|}},$$
\eqref{kernel} becomes
$$-\frac{1+\varepsilon}{\pi  \frac{1}{\log(r^2)} \log(r^2)}+\pi^{-1} \sum _{j=0}^{\infty } \left(\frac{r^{2+2 j}}{\left(-r^{2+2 j}+\frac{1}{(1+\varepsilon)\log(r^2)}\right)^2}+\frac{r^{2 j}}{\left(1-r^{2
j} \frac{1}{(1+\varepsilon)\log(r^2)}\right)^2}\right).$$ One of
course has to check that $r<|z|,|\zeta|<1$, for sufficiently small
$r$, to ensure that this special pair of points belongs to the
annulus. This is obvious. Now consequently using the negativity of
$\log(r)$, \eqref{eq2} and \eqref{eq3} one has
\begin{equation}\label{tag1} -\frac{1+\varepsilon}{\pi}+\pi^{-1} \sum _{j=0}^{\infty } \left(\frac{r^{2+2 j}(1+\varepsilon)^2(\log(r^2))^2}{\left(1-r^{2+2 j}(1+\varepsilon)\log(r^2)\right)^2}+\frac{r^{2 j}}{\left(1-r^{2
j} \frac{1}{(1+\varepsilon)\log(r^2)}\right)^2}\right)\leq\end{equation}
$$-\frac{1+\varepsilon}{\pi}+\pi^{-1} \sum _{j=0}^{\infty } \left(r^{2 j}(1+\varepsilon)^2\varepsilon^2+r^{2 j}\right)=$$
$$-\frac{1+\varepsilon}{\pi}+\frac{1}{\pi}((1+\varepsilon)^2\varepsilon^2+1)+\frac{1}{\pi} \frac{r^{2}}{1-r^2}((1+\varepsilon)^2\varepsilon^2+1)\leq$$
$$\frac{-\varepsilon+((1+\varepsilon)^2(\varepsilon^{2}+1)+1)\varepsilon^2}{\pi}.$$
Clearly this is negative for sufficiently small $\varepsilon$.

Similarly for the special choice
$$z=\frac{1}{\sqrt{|\log(r^2)|}},\zeta=\frac{-1}{(1-\varepsilon)\sqrt{|\log(r^2)|}},$$
which is also good for small $r$, \eqref{kernel} becomes
\begin{equation}\label{tag2}-\frac{1-\varepsilon}{\pi  \frac{1}{\log(r^2)} \log(r^2)}+\pi^{-1} \sum _{j=0}^{\infty } \left(\frac{r^{2+2 j}}{\left(-r^{2+2 j}+\frac{1}{(1-\varepsilon)\log(r^2)}\right)^2}+\frac{r^{2 j}}{\left(1-r^{2
j} \frac{1}{(1-\varepsilon)\log(r^2)}\right)^2}\right)=\end{equation}

$$-\frac{1-\varepsilon}{\pi}+\pi^{-1} \sum _{j=0}^{\infty } \left(\frac{r^{2+2 j}(1-\varepsilon)^2(\log(r^2))^2}{\left(1-r^{2+2 j}(1-\varepsilon)\log(r^2)\right)^2}+\frac{r^{2 j}}{\left(1-r^{2
j} \frac{1}{(1-\varepsilon)\log(r^2)}\right)^2}\right)\geq$$
$$-\frac{1-\varepsilon}{\pi}+\frac{1}{\pi}
\frac{r^{0}}{\left(1-r^{0 }
\frac{1}{(1-\varepsilon)\log(r^2)}\right)^2}\geq$$
$$-\frac{1-\varepsilon}{\pi}+\frac{1}{\pi}
\frac{(1-\varepsilon)^2}{\left(
1-\varepsilon+\varepsilon^2\right)^2},$$ by \eqref{eq1}. Now
expanding into Taylor series gives one
$\left(\frac{1-x}{1-x+x^2}\right)^2=1-2x^2+ o(x^2)$, hence our
expression is approximately
$$\frac{\varepsilon-2\varepsilon^2}{\pi}>0,$$ for sufficiently small $\varepsilon$.

So $K(z,\zeta)$ is real for $z=\frac{1}{\sqrt{|\log(r^2)|}}$ and $\zeta$
from the interval
$$\left[\frac{-1}{(1-\varepsilon)\sqrt{|\log(r^2)|}},\frac{-1}{(1+\varepsilon)\sqrt{|\log(r^2)|}}\right]$$
and has different sign on the endpoints of this interval.
Therefore it must have a zero there.

To compute the Bergman distance between $z$ and $\zeta$ one has to
find the Bergman metric first
\begin{equation}\label{metric}\beta^2(z)=\frac{\partial^2\log K(z,z)}{\partial z\partial
\overline{z}}=\frac{K(z,z)_{1\bar 1} K(z,z)-K(z,z)_{1}K(z,z)_{\bar 1}}{K(z,z)^2}.\end{equation}

Since this expression is invariant under rotations it will be
enough to compute $\beta(z)$, for
$z=\frac{1}{c\sqrt{|\log(r^2)|}}\in \mathbb R, c$ close to $1$ or
$-1$. The expressions involved in formula \eqref{metric} are

$$K(z,z)=-\frac{1}{\pi  z\overline{z} \log(r^2)}+\pi^{-1} \sum _{j=0}^{\infty } \left(\frac{r^{2+2 j}}{\left(-r^{2+2 j}+z\overline{z}\right)^2}+\frac{r^{2 j}}{\left(1-r^{2
j} z\overline{z}\right)^2}\right)$$
$$K(z,z)_{1}=\frac{1}{\pi  z^2\overline{z} \log(r^2)}+\pi^{-1} \sum _{j=0}^{\infty } \left(-\frac{2r^{2+2 j}\overline{z}}{\left(-r^{2+2 j}+z\overline{z}\right)^3}+\frac{2r^{4j}\overline{z}}{\left(1-r^{2
j} z\overline{z}\right)^3}\right)$$
$$K(z,z)_{\bar{1}}=\frac{1}{\pi  z\overline{z}^2 \log(r^2)}+\pi^{-1} \sum _{j=0}^{\infty } \left(-\frac{2r^{2+2 j}\overline{z}}{\left(-r^{2+2 j}+z\overline{z}\right)^3}+\frac{2r^{4j}\overline{z}}{\left(1-r^{2
j} z\overline{z}\right)^3}\right)$$
$$K(z,z)_{1\bar 1}=-\frac{1}{\pi  (z\overline{z})^2 \log(r^2)}+\pi ^{-1}\sum _{j=0}^{\infty } \Big(
\frac{6r^{2+2j}z\overline{z}}{\left(-r^{2+2
j}+z\overline{z}\right)^4}-
\frac{2r^{2+2j}}{\left(-r^{2+2j}+z\overline{z}\right)^3}+$$
$$\frac{6r^{6j}z\overline{z}}{\left(1-r^{2j}z\overline{z}\right)^4}+
\frac{2r^{4j}}{\left(1-r^{2j}z\overline{z}\right)^3}\Big).$$

The expressions seem complicated, however almost every summand in
the series above is negligible. To show this one proceeds similarly as in \eqref{tag1}, \eqref{tag2}.

\begin{equation}\label{ABCD}\left|\frac{Ar^{B}\left(\frac{1}{c\sqrt{|\log(r^2)|}}\right)^{D}}{\left(-r^{B}-\frac{1}{c^2\log(r^2)}\right)^F}\right|=
\left| \frac{Ar^B}{(c\sqrt{|\log(r^2)|})^{D-2F}(1+c^2\log(r^2)r^B)^F)}\right|\leq\end{equation}
$$\leq\left| \frac{Ar^B}{(c\sqrt{|\log(r^2)|})^{D-2F}(1-c^2\varepsilon)^{F}}\right|,$$
when $B>0$.

The other terms are estimated by
\begin{equation}\label{ABCDE}\left|\frac{A'r^{B'}\left(\frac{1}{c\sqrt{|\log(r^2)|}}\right)^{D'}}{ \left( 1+r^{E'}\frac{1}{c^2\log(r^2)}\right)^{F'}}\right|\leq\left|\frac{A'r^{B'}} {\left(c\sqrt{|\log(r^2)|}\right)^{D'}\left(1-\frac{\varepsilon^{2}}{c^2}\right)^{F'} }\right|.\end{equation}

Now because $r^{B}\log(r)^{E}$ tends to zero for positive $B$ and
arbitrary $E$ one can employ either of the estimates
\eqref{eq1},\eqref{eq2} and finally come with estimate of the sort
$H|r|^{G}$. The only exception is clearly when $B$ or $B'=0$ i.e.,
$j=0$,  so adding up one gets a geometric power control on the
series.

Now $$K\left(\frac{1}{c\sqrt{|\log(r^2)|}},\frac{1}{c\sqrt{|\log(r^2)|}}\right)= \frac{c^2}{\pi }+\frac{1}{\pi  \left(1+\frac{1}{ c^2
\log(r^2)}\right)^2}+o(C)=\frac{c^2+1}{\pi }+o(C),
$$
$$K_{1}=K\left(\frac{1}{c\sqrt{|\log(r^2)|}},\frac{1}{c\sqrt{|\log(r^2)|}}\right)_{\bar 1}= -\frac{ c^3 \sqrt{|\log(r^2)|}}{\pi
}+o(C),
$$
$$K\left(\frac{1}{c\sqrt{|\log(r^2)|}},\frac{1}{c\sqrt{|\log(r^2)|}}\right)_{1\bar 1}= \frac{2}{\pi  \left(1+\frac{1}{
c^2 \log(r^2)}\right)^3}-\frac{ c^4 \log(r^2)}{\pi
}+o(C)=$$$$=\frac{2-c^4\log(r^2)}{\pi   }+o(C),$$ where the convention
is $o(C)\equiv o(const)$.

Finally \eqref{metric} becomes
$$\beta=\sqrt{\frac{(c^2+1+o(C))(2-c^4\log(r^2)+o(C))- ( c^3 \sqrt{|\log(r^2)|}
+o(C))^2}{(c^2+1+o(C))^2}}.$$
 The path which approximates the
distance is as follows. First one joins $\zeta$ with the point
$\frac{-1}{\sqrt{|\log(r^2)|}}$ via a linear segment. Then this
point is joined with $z=\frac{1}{\sqrt{|\log(r^2)|}}$ via the
half-circle
$$[0,1]\ni t\to e^{(\pi-\pi t)i} \frac{1}{\sqrt{|\log(r^2)|}}.$$ The segment will be denoted by $\gamma_{1}$ and the half-circle by $\gamma_{2}$.

The geodesics of the Bergman metric in the annulus are classified
in \cite{MR709860} and one easily sees that our path is not a geodesic,
however the integral distance over it is a close enough
approximation.

The  integrals that one has to assume are
$$I_{1}:=\int_{0}^{1}\beta(\gamma_{1}(t))\left|\frac{\partial\gamma_{1}}{\partial t}(t)\right|dt$$
$$I_{2}:=\int_{0}^{1}\beta(\gamma_{2}(t))\left|\frac{\partial\gamma_{2}}{\partial t}(t)\right|dt$$

Let $\zeta=\frac{-1}{s\sqrt{|\log(r^2)|}}, s\in[1-\varepsilon,1+\varepsilon]$.
The parametrization of $\gamma_{1}$ will be

$$[0,1]\ni t\to \frac{-1}{(s+t(1-s))\sqrt{|\log(r^2)|}}.$$
After straightforward computations one obtains
$$\int_{0}^{1}\lim_{r\to 0^{+}}\beta(\gamma_{2}(t))\left|\frac{\partial\gamma_{2}}{\partial t}(t)\right|dt=$$
$$\int_{0}^{1}|s-1|\sqrt{\frac{c^4}{(c^2+1)^{2}c^2}} dt,$$
where $c=s+t(1-s)$. Now since $|s-1|<\varepsilon$ and since the
integral is clearly finite we conclude that
$$\lim_{\varepsilon\to 0^{+}}\lim_{r\to0^{+}}I_{1}=0.$$

$$\lim_{r\to 0^{+}} I_{2}(r)=\int_{0}^{1}\lim_{r\to 0^{+}}\beta(\gamma_{2}(t))\left|\frac{\partial\gamma_{2}}{\partial t}(t)\right|dt$$

In this case $c=1$, $\beta$ is constant on $\gamma_{2}$ and equals

$$\sqrt{\frac{(2+o(C))(2-\log(r^2)+o(C))- (  \sqrt{|\log(r^2)|}
+o(C))^2}{(2+o(C))^2}}.$$ Further
$\frac{\partial\gamma_{2}}{\partial t}=-\pi i e^{(\pi-\pi t)i}
\frac{1}{\sqrt{|\log(r^2)|}}$, hence
$$\beta(\gamma_{2}(t))\left|\frac{\partial\gamma_{2}}{\partial t}(t)\right|\approx\pi\sqrt{\frac{2(2-\log(r^2))+ \log(r^2)
}{4}\frac{1}{|\log(r^2)|}} \to \frac{\pi}{2},$$
when $r\to 0$.

Now the Bergman distance between $z$ and $\zeta$ when $r\to0$ is bounded between $\frac{\pi}{2}$ by Theorem \ref{wheredefined} and $I_{1}+I_{2}$ and hence also tends to $\frac{\pi}{2}$. This proves Theorem \ref{annulusestimate1}.

\bigskip

Now we provide the example proving Theorem\ref{annulusestimate2}. As above we study the circular annulus $P_{r}$. To prove non-immersivity, following \eqref{arccos} and \eqref{wherebad} one has to localize the zeros of 
$\det K^{2}(z,\zeta)\frac{\partial^{2}}{\partial z\partial\bar\zeta}\log K(z,\zeta)$, that is to say - of 
\begin{equation}\label{expression1}
K(z,\zeta)\frac{\partial^{2}}{\partial z\partial\bar\zeta} K(z,\zeta)-\frac{\partial}{\partial z} K(z,\zeta)\frac{\partial}{\partial\bar\zeta} K(z,\zeta)
 \end{equation}
Fix $0<\varepsilon<<1$. Using arguments similar to \ref{ABCD},\ref{ABCDE} one can consider only the terms not containing an $r$ to a positive power in the expansions of the above objects. That is to say

\begin{equation}\label{equat1}
 K(z,\zeta)\approx -\frac{1}{\pi}\frac{1}{z\bar\zeta\log(r^2)}+\frac{1}{\pi}\frac{1}{(1-z\bar\zeta)^{2}}
\end{equation}

\begin{equation}\label{equat2}
 \frac{\partial}{\partial z}K(z,\zeta)\approx \frac{1}{\pi}\frac{1}{z^{2}\bar\zeta\log(r^2)}+\frac{1}{\pi}\frac{2\bar\zeta}{(1-z\bar\zeta)^{3}}
\end{equation}
\begin{equation}\label{equat3}
 \frac{\partial}{\partial\bar\zeta}K(z,\zeta)\approx \frac{1}{\pi}\frac{1}{z\bar\zeta^{2}\log(r^2)}+\frac{1}{\pi}\frac{2z}{(1-z\bar\zeta)^{3}}
\end{equation}
\begin{equation}\label{equat4}
\frac{\partial^{2}}{\partial z\partial\bar\zeta} K(z,\zeta)\approx -\frac{1}{\pi}\frac{1}{z^{2}\bar\zeta^{2}\log(r^2)}+\frac{1}{\pi}\frac{6z\bar\zeta}{(1-z\bar\zeta)^{4}}+\frac{1}{\pi}\frac{2}{(1-z\bar\zeta)^{3}},
\end{equation}
of course for a suitable choice of $z$ and $\zeta$. In our case we fix $\zeta$ to be $\frac{1}{\sqrt[4]{|2\log(r^2)|}}$ and put $z=\frac{i}{\xi\sqrt[4]{|2\log(r^2)|}}$, where $\xi$ is an independent of $r$ complex variable presumably very close to $1$.

Plugging these values in the expressions \eqref{equat1},\eqref{equat2},\eqref{equat3},\eqref{equat4}, one sees that the dominant term in \eqref{equat1} will be $\frac{1}{\pi}\frac{1}{(1-z\bar\zeta)^{2}}$, i.e., $\frac{1}{\pi}\frac{1}{\left(1-\frac{i}{\xi\sqrt{|2\log(r^2)|}}\right)^{2}}\longrightarrow \frac{1}{\pi}$, when $r\to 0$. Both summands in both expressions \eqref{equat2} and \eqref{equat3} have the same asymptotic behaviour $\sim \frac{const}{\sqrt[4]{|\log(r^2)|}}\to 0$. Finally in \eqref{equat4} the first and the last terms are dominating, summing up to
$$-\frac{2\xi^{2}}{\pi}+\frac{1}{\pi}\frac{2}{\left(1-\frac{i}{\xi\sqrt{|2\log(r^2)|}}\right)^{3}}\longrightarrow \frac{2-2\xi^{2}}{\pi}.$$

Back to expression \eqref{expression1} one sees that it can be written as $F(z,\zeta)+G(z,\zeta)$, where
$$F(z,\zeta)=\frac{1}{\pi}\frac{1}{(1-z\bar\zeta)^{2}}\left[-\frac{1}{\pi}\frac{1}{z^{2}\bar\zeta^{2}\log(r^2)}+\frac{1}{\pi}\frac{2}{(1-z\bar\zeta)^{3}}\right]$$
and $G(z,\zeta)$ is the sum of all the other expressions.
By taking $\xi$ on a circle of radius $\varepsilon$ around $1$ one sees that
$$\lim_{r\to 0^{+}}\left|F\left(\frac{i}{\xi\sqrt[4]{|2\log(r^2)|}},\frac{1}{\sqrt[4]{|2\log(r^2)|}}\right)\right|=\left|\frac{2-2\xi^{2}}{\pi}\right|>0,$$
whereas
$$\lim_{r\to 0^{+}}\left|G\left(\frac{i}{\xi\sqrt[4]{|2\log(r^2)|}},\frac{1}{\sqrt[4]{|2\log(r^2)|}}\right)\right|=0.$$

By Rouch\'{e}'s theorem
$$K(z,\zeta)\frac{\partial^{2}}{\partial z\partial\bar\zeta} K(z,\zeta)-\frac{\partial}{\partial z} K(z,\zeta)\frac{\partial}{\partial\bar\zeta} K(z,\zeta)_{\Big|z=\frac{i}{\xi\sqrt[4]{|2\log(r^2)|}},\zeta=\frac{1}{\sqrt[4]{|2\log(r^2)|}}},$$
as a holomorphic function of $\xi$, has the same number of zeros in $\{z: |z-1|<\varepsilon\}$ as $$F(z,\zeta)_{\Big|z=\frac{i}{\xi\sqrt[4]{|2\log(r^2)|}},\zeta=\frac{1}{\sqrt[4]{|2\log(r^2)|}}},$$

provided that $r$ is sufficiently small. Now solving 
$$\frac{2}{\left(1-\frac{i}{\xi\sqrt{|2\log(r^2)|}}\right)^{3}}=2\xi^{2}$$ one can check that
$$\frac{i}{c}+\frac{(1+i\sqrt{3})c^3}{\sqrt[3]{12(-9ic^{8}+\sqrt{3}\sqrt{-27c^{16}-4c^{18}})}}+\frac{(1-i\sqrt{3})
\sqrt[3]{-9ic^{8}+\sqrt{3}\sqrt{-27c^{16}-4c^{18}}}}{c^{3}2\sqrt[3]{18}},$$
where $c=\sqrt{|2\log(r^{2})|}$ is a solution of this equation, which lies in $\{z: |z-1|<\varepsilon\}$ for sufficiently small $r$.

Next we sketch the rest of the proof without going into the (very technical) calculational details. Once locallized, the zeros of $\frac{\partial^{2}}{\partial z\partial \bar \zeta}\log K(z,\zeta)$ are joined by a path $\tilde \gamma_{1}\cup\tilde \gamma_{2}$, $\tilde \gamma_{1}$ being the linear segment joining $z$ with the point $\frac{i}{\sqrt[4]{|2\log(r^2)|}}$. As in the proof of Theorem\ref{annulusestimate1} the integral distance over $\tilde \gamma_{1}$ with respect to the metric $\tilde\beta(.,\circ)$ tends to $0$ when $r\to 0$ and $\varepsilon\to 0$.
 Now $\tilde \gamma_{2}$ will be the arc $$[0,1]\ni t\longrightarrow \frac{1}{\sqrt[4]{|2\log(r^2)|}}e^{\frac{1-t}{2}\pi i}.$$
Again we use the fact that $\tilde \beta$ is constant in the first variable along this arc, by conformal invariance. It is therefore enough to compute the metric tensor of $\tilde \beta$ only at the point $\frac{1}{\sqrt[4]{|2\log(r^2)|}}=\tilde \gamma_{2}(1)$.
Recall that
$$\tilde\beta^{2}(\zeta,X)=\left[2\frac{\partial^{2}}{\partial \zeta\partial \bar \zeta}K(\zeta,\zeta)+\frac{\partial^{2}}{\partial \zeta\partial \bar \zeta}\log\frac{\partial^{2}}{\partial \zeta\partial \bar \zeta}K(\zeta,\zeta)\right]|X|^2$$

and hence one has to compute

$$\frac{\partial^{2}}{\partial \zeta\partial \bar \zeta}\log \left[K(\zeta,\zeta)^2 \frac{\partial^{2}}{\partial \zeta\partial \bar \zeta}K(\zeta,\zeta)\right]=$$$$\frac{\partial^{2}}{\partial \zeta\partial \bar \zeta}\log \left[K(\zeta,\zeta)\frac{\partial^{2}}{\partial \zeta\partial \bar \zeta}K(\zeta,\zeta)-\frac{\partial}{\partial \zeta}K(\zeta,\zeta)\frac{\partial}{\partial \bar\zeta}K(\zeta,\zeta)\right]=$$
\begin{equation}\label{finall}\frac{K(\zeta,\zeta)_{1}K(\zeta,\zeta)_{\bar1}K(\zeta,\zeta)_{11}K(\zeta,\zeta)_{\bar1\bar1}}{D^2}-\frac{K(\zeta,\zeta)_{11}K(\zeta,\zeta)_{\bar1\bar1}}{D}+\end{equation}
$$\frac{K(\zeta,\zeta)K(\zeta,\zeta)_{\bar1}K(\zeta,\zeta)_{1\bar1\bar1}K(\zeta,\zeta)_{11}}{D^{2}}+\frac{K(\zeta,\zeta)K(\zeta,\zeta)_{1}K(\zeta,\zeta)_{11\bar1}K(\zeta,\zeta)_{\bar1\bar1}}{D^{2}}-$$
$$-\frac{K(\zeta,\zeta)^{2}K(\zeta,\zeta)_{1\bar1\bar1}K(\zeta,\zeta)_{11\bar1}}{D^{2}}+\frac{K(\zeta,\zeta)K(\zeta,\zeta)_{11\bar1\bar1}}{D},$$
where the denominator $D=K(\zeta,\zeta)K(\zeta,\zeta)_{1\bar1}-K(\zeta,\zeta)_{1}K(\zeta,\zeta)_{\bar1}$.
We first obtain the asymptotic of $D$. Much of the analysis from \eqref{equat1}-\eqref{equat4} can be repeated to show that
$$K(\zeta,\zeta)\approx\frac{1}{\pi}$$
$$K(\zeta,\zeta)_{1}=K(\zeta,\zeta)_{\bar1}, \text{ tends to } 0 \text{ faster than } \frac{1}{\sqrt[4]{|\log(r^2)|}}$$
$$K(\zeta,\zeta)_{1\bar 1}\approx\frac{4}{\pi}$$

The change of sign from $\approx\frac{2-2\xi^{2}}{\pi}$ to $\approx\frac{4}{\pi}$ is due to the absence of $i$ in the value of $z=\zeta$. Hence $D\to\frac{4}{\pi^{2}}$.

Similarly it can be shown that
$$K(\zeta,\zeta)_{11}=K(\zeta,\zeta)_{\bar1\bar1}\approx\frac{4}{\pi}$$
$$K(\zeta,\zeta)_{11\bar1}=K(\zeta,\zeta)_{1\bar1\bar1}\approx-\frac{4\sqrt[4]{2}\sqrt[4]{|\log(r^2)|}}{\pi}$$
$$K(\zeta,\zeta)_{11\bar1\bar1}\approx\frac{8\sqrt{2}\sqrt{|\log(r^2)|}}{\pi}$$

This gives one that only the last two terms in the expression \eqref{finall} are relevant in the asymptotic behaviour of the metric tensor $\tilde\beta$ which is
$$\approx -\frac{\left(\frac{1}{\pi}\right)^{2}\frac{(-4\sqrt[4]{2})^{2}\sqrt{|\log(r^2)|}}{\pi^{2}}}{\left(\frac{4}{\pi^{2}}\right)^{2}}+\frac{\frac{1}{\pi}\frac{8\sqrt{2}\sqrt{|\log(r^2)|}}{\pi}}{\frac{4}{\pi^{2}}}=\sqrt{2}\sqrt{|\log(r^2)|}.$$

Now $\frac{\partial}{\partial t}\tilde\gamma_{2}(t)=\frac{1}{\sqrt[4]{|2\log(r^2)|}}e^{\frac{1-t}{2}\pi i}\left(-\frac{1}{2}\pi i\right)$. Finally the integral distance over $\tilde\gamma_{2}$ with respect to $\tilde \beta$ is
\begin{equation}\label{theend}\int_{0}^{1}\sqrt{\tilde{\beta}\left(\tilde\gamma_{2}(t),\frac{\partial}{\partial t}\tilde\gamma_{2}(t)\right)}\approx\int_{0}^{1}\sqrt{\sqrt{2}\sqrt{|\log(r^2)|}\frac{1}{\sqrt{|2\log(r^2)|}}\frac{\pi^{2}}{4}}\to \frac{\pi}{2},\end{equation}
when $r\to 0$. Now by \eqref{finalll}, the estimate of the distance over $\tilde\gamma_{1}$ and \eqref{theend}, one has that   $\tilde{dist}(z,\zeta)\to\frac{\pi}{2}$, which establishes the claim.

\end{section}
\bibliographystyle{amsplain.bst}
\bibliography{bergmandistance}

\end{document}